\documentclass[english]{article}
\usepackage{preamble}
\usepackage{comment}

\begin{document}

\title{On the Strict Picard Spectrum of Commutative Ring Spectra}
\author{Shachar Carmeli}
\date{\today}
\maketitle

\abstract{
We compute the connective spectra of maps from $\ZZ$ to the Picard spectra of the spherical Witt vectors associated with perfect rings of characteristic $p$. As an application, we determine the connective spectrum of maps from $\ZZ$ to the Picard spectrum of the sphere spectrum. 
}
\tableofcontents{}

\section{Introduction}
\subsection{Background}
Picard groups play a central role in algebraic geometry and number theory. Among other things, they are useful in studying unramified abelian Galois extensions of number and function fields and embeddings of varieties into projective spaces. The Picard group of a scheme $S$, consisting of invertible quasi-coherent $\mathcal{O}_S$-modules, is closely related to the sheaf of units in $\mathcal{O}_S$, via the formula 
\begin{equation}\label{eq:Pic_H^1}
\Pic(S)= H^1(S,\mathcal{O}_S^\times).
\end{equation}

In higher algebra, the $\infty$-category of connective spectra provides a natural generalization of abelian groups, and commutative ring spectra generalize classical commutative rings. Accordingly, to a commutative ring spectrum $R$, Hopkins associated a connective spectrum $\pic(R)$ of $\otimes$-invertible $R$-modules. Conveniently, in the higher algebraic framework, the relationship between the Picard group and the units, indicated in \Cref{eq:Pic_H^1}, becomes more transparent. Namely, the units in $R$ are obtained from the Picard spectrum of $R$ by applying the loops functor: 
\[
\Omega\pic(R) \simeq R^\times. 
\]
As a result, the positive homotopy groups of $\pic(R)$ are easily computable in terms of $R$ itself, as $\pi_1\pic(R)\simeq \pi_0R^\times$ and $\pi_i\pic(R) = \pi_{i-1}R$ for $i\ge 2$.

While the Picard spectra of commutative ring spectra are the most direct analogs of the Picard groups of commutative rings, there are features of the Picard groups that do not generalize well. For example, if $\mathcal{L}$ is an invertible $\mathcal{O}_S$-module over a scheme $S$, the tensor algebra 
\[
T(\mathcal{L}) = \mathcal{O}_S \oplus \mathcal{L}\oplus \mathcal{L}^{\otimes 2} \oplus \dots
\] is commutative, and the relative spectrum $\Spec(T(\mathcal{L}))$ is the total space of the dual line bundle $\mathcal{L}^\vee$. This way, one can treat elements of the Picard group both as algebraic objects (invertible sheaves) or as geometric objects (line bundles). 

When passing to the higher categorical realm, the commutativity of the tensor algebra is no longer automatic. The key \emph{property} of invertible sheaves responsible for this commutativity is the triviality of the permutation action of the symmetric group $\Sigma_n$ on $\mathcal{L}^{\otimes n}$. However, for modules over a commutative ring spectrum, the triviality of this action is not a property, but an \emph{additional structure}, which might not be unique or even not exist.  

One way to remedy this deficiency is to consider only invertible modules \emph{together with} a coherent trivialization of these permutation actions. Moreover, there is a convenient way to encode these trivializations and their interaction with the tensor product. Namely, the discrete spectrum $\ZZ$ is generated by a single ''strictly commutative" element, and so one can single out elements of the Picard spectrum with such strict commutative behavior as images of maps from $\ZZ$. This leads to the following notion. 
\begin{defn}
Let $R$ be a commutative ring spectrum. A \emph{strictly invertible $R$-module} is a morphism of spectra $\ZZ \to \pic(R)$. 
\end{defn}
The collection of strictly invertible $R$-modules identifies with the set of connected components of a connective spectrum $\spic(R)$, the strict Picard spectrum of $R$,  which provides a canonical delooping of the strict units spectrum $\GG_m(R)$. 
For strictly invertible modules, one recovers the familiar relationship between line bundles and invertible sheaves. Indeed, a strictly invertible $R$-module can be seen as an $\EE_\infty$-map 
\[
\mathcal{L}\colon\ZZ \to \Omega^\infty\pic(R)\subseteq \Mod_R^\simeq,
\] 
for which we can form the ''Thom object" (as in \cite{ando2014categorical}) over the natural numbers $\NN\subseteq\ZZ$: 
\[\Th_{\NN}(\mathcal{L}):=\underset{\NN}{\colim}\mathcal{L}.
\] 
This $R$-module admits a canonical commutative $R$-algebra structure, generalizing to higher algebra the commutative algebra structure on the tensor algebra of an invertible sheaf.

Despite its conceptual advantages, it is generally hard to compute the homotopy groupts of $\spic(R)$. So far, few cases have been worked out beyond discrete commutative ring spectra. An interesting example arising from chromatic homotopy theory has been computed recently by Burklund, Schlank, and Yuan: 

\begin{othm}[{{\cite[Theorem 8.17]{burklund2022chromatic}}}] 
Let $E_n$ be the Lubin-Tate theory associated with a formal group law of height $n>0$ over an algebraically closed field $\kappa$ of characteristic $p$. Then, 
\[
\spic(E_n)\simeq \Sigma^{n+2}\ZZ_p \oplus \Sigma \kappa^\times. \footnote{The resulting formula $\GG_m(E_n)\simeq \kappa^\times \oplus \Sigma^{n+1}\ZZ_p$ was conjectured, and later proved, by Hopkins and Lurie.}
\]
\end{othm}

The initial example of a commutative ring spectrum is the sphere spectrum $\Sph$. This paper is dedicated to the computation of the strict Picard spectrum of the sphere spectrum, its $p$-completions for various primes $p$, and other commutative ring spectra of a similar flavor. 
\subsection{Main Results} 

Our analysis of the strict Picard spectrum starts in the $p$-complete world. Recall that, for a perfect ring $\kappa$ of characteristic $p$, one associates a $p$-complete connective commutative ring spectrum $\Sph\WW(\kappa)$, known as the ring of \emph{spherical Witt vectors} (see \cite[Example 5.2.7]{lurie2018elliptic}). They  generalize the $p$-complete sphere in that $\Sph\WW(\FF_p)\simeq \Sph_p$. Our first result is a complete determination of the strict Picard spectrum of these commutative ring spectra in terms of classical invariants of $\kappa$. 
\begin{theorem} [\Cref{strict_picard_spherical_witt}] \label{intro:spherical_witt}
Let $\kappa$ be a perfect ring of characteristic $p$, and denote by $\Cl(\kappa)$ the Picard group of the abelian category $\Mod_\kappa^\heartsuit$ of discrete $\kappa$-modules. The strict Picard spectrum of the spherical Witt vectors $\Sph\WW(\kappa)$
is given by
\[
\spic(\Sph\WW(\kappa))\simeq \Cl(\kappa) \oplus \Sigma \kappa^\times. 
\]
\end{theorem}
Here and in the sequel, all the abelian groups on the right-hand side should be regarded as discrete connective spectra.
By taking $\kappa = \FF_p$, we obtain:
\begin{cor}
The strict Picard spectrum of the $p$-complete sphere is given by 
\[
\spic(\Sph_p) \simeq \Sigma\FF_p^\times.
\]
\end{cor}
Our second result is a computation of the strict Picard spectrum of the sphere spectrum. Let $\widehat{\ZZ}\simeq \prod_p \ZZ_p$ be the profinite completion of the integers.  
\begin{theorem}[\Cref{strict_picard_sphere}]\label{intro:strict_sphere}
The strict Picard spectrum of the sphere spectrum is given by
\[
\spic(\Sph)\simeq \widehat{\ZZ}. 
\]
\end{theorem}

\begin{rem}
By taking loops, we can now compute the strict units in the commutative ring spectra $\Sph\WW(\kappa)$ and $\Sph$. Namely, we deduce that 
\[
\GG_m(\Sph\WW(\kappa))\simeq \kappa^\times
\]
and
\[
\GG_m(\Sph)\simeq 0.
\]
Some progress on the computation of $\GG_m(\Sph)$, using the symmetric power filtration of $\ZZ$ and the corresponding spectral sequence, was made by Fung in his Ph.D. thesis \cite{fung2020strict}.
\end{rem}

\subsection{Outline of the Proofs} 
It is a classical observation that for a perfect ring $\kappa$ of characteristic $p$, the map 
\[
p\colon \Cl(\kappa) \to \Cl(\kappa)
\]
is an isomorphism. 
This follows from the fact that, immediately from \Cref{eq:Pic_H^1}, the Frobenius isomorphism $\varphi\colon \kappa \iso \kappa$ acts on the Picard group of $\kappa$ via multiplication by $p$. 
More generally, for an arbitrary commutative ring $R$, the mod $p$ reduction map $\can\colon R\to R/p$ and the $p$-th power map $\varphi\colon R\to R/p$ satisfy  
\[
\can_*(\mathcal{L})^{\otimes p}\simeq \varphi_*(\mathcal{L})
\]
for every invertible $R$-module $\mathcal{L}$.

The first step in proving our main theorems is the observation that all the ingredients in this argument have higher-algebraic analogs. Namely,
\begin{itemize}
    \item The quotient $R/p$ corresponds to the \emph{Tate construction} $R^{tC_p}$.
    \item The reduction mod $p$ map $R\to R/p$ corresponds to the \emph{canonical map} $\can\colon R\to R^{tC_p}$. 
    \item The $p$-th power map corresponds to the \emph{Tate-valued Frobenius} $\fr\colon R\to R^{tC_p}$, defined by Nikolaus and Scholze in \cite{nikolaus2018topological}.
\end{itemize}
 
However, unlike in the classical case, the identity $\fr_*(\mathcal{L}) \simeq \can_*(\mathcal{L})^{\otimes p}$ \emph{does not hold} for a general invertible $R$-module $\mathcal{L}$. In fact, $\can_*(\mathcal{L})^{\otimes p}$ is given by the Tate construction of the group $C_p$ acting trivially on $\mathcal{L}^{\otimes p}$, while $\fr_*(\mathcal{L})$ is the Tate construction for the cyclic permutation action on the tensor factors. 

A key insight, which allows us to generalize the argument above from perfect rings to their spherical Witt vectors, is that the identity $\fr_*(\mathcal{L}) \simeq \can_*(\mathcal{L})^{\otimes p}$ \emph{does hold} if $\mathcal{L}$ is strictly invertible. Indeed, the data of such an element include, in particular, a trivialization of the cyclic permutation action on $\mathcal{L}^{\otimes p}$, and hence provide the desired identification. 
This consideration shows that $p$ is invertible on the strict Picard spectrum of a commutative ring spectrum $R$, provided that the Tate-valued Frobenius $\varphi^p\colon R\to R^{tC_p}$ is an isomorphism. In analogy with the classical story, we call such commutative ring spectra \emph{perfect}.  

For a perfect ring $\kappa$ of characteristic $p$, the commutative ring spectrum $\Sph\WW(\kappa)$ is perfect\footnote{This essentially follows from the  Segal Conjecture for cyclic groups, proved by Lin and Gunawardena, see \cite{lin1980conjectures} and \cite{gunawardena1980segal}.}, so the consideration above applies to it. 
Since $\Sph\WW(\kappa)$ is $p$-complete, the homotopy groups of $\pic(\Sph\WW(\kappa))$ are $p$-complete starting from $\pi_2$, and hence the invertibility of $p$ on the strict Picard spectrum shows that only the first two homotopy groups of  $\pic(\Sph\WW(\kappa))$ are involved in the formation of the strict Picard spectrum. This reduces the determination of $\spic(\Sph\WW(\kappa))$ to a finite computation with $\Pic(\Sph\WW(\kappa))$ and $\pi_0\Sph\WW(\kappa)^\times$. We then calculate these two groups using the corresponding invariants of $\kappa$ and prove \Cref{intro:spherical_witt}. 

To deduce \Cref{intro:strict_sphere} from \Cref{intro:spherical_witt}, we use the arithmetic fracture square. Roughly, this square shows how the sphere spectrum $\Sph$ is glued from its various $p$-completions $\Sph_p$ and its rationalization $\QQ$, along the ring of finite ad\'{e}les $\mathbb{A}$. Using this square, we reduce the calculation of $\spic(\Sph)$ to that of $\spic(\Sph_p)$, which was computed in \Cref{intro:spherical_witt}, and the computation of the spectra of maps from $\ZZ$ to $\QQ^\times$ and $\mathbb{A}^\times$, which is elementary.

\subsection{Organization}
In section 2, we introduce the canonical map and the Tate-valued Frobenius. We then compute the extension of scalars of perfect modules along these two maps. 

In section 3, we discuss the strict Picard spectrum. We recall the definition of the Picard spectrum, study its strict variant via the formalism of ''strict elements" in a connective spectrum, and explain how these strict elements in the Picard spectrum trivialize the cyclic actions on their tensor powers.  

In section 4, we consider the strict Picard spectra of commutative ring spectra, study their behavior under scalar extension along the Frobenius and canonical maps, and prove the main results of this paper: \Cref{intro:spherical_witt} and \Cref{intro:strict_sphere}.  
\subsection{Conventions}
\begin{itemize}
    \item By a commutative ring spectrum we mean a commutative (a.k.a. $\EE_\infty$) algebra in the symmetric monoidal $\infty$-category $\Sp$. For such a commutative ring spectrum $R$, we denote by $\Mod_R$ the $\infty$-category of $R$-module spectra and $\Perf_R$ the full subcategory spanned by the perfect (a.k.a. compact) $R$-modules.  
    \item For an object $X$ in a symmetric monoidal $\infty$-category, we denote by $\pi_i(X)$ the $i$-th homotopy group of the space of maps from the unit to $X$. 
    \item For a map of spaces $f\colon A\to B$ and an $\infty$-category $\cC$, we denote by $f^*\colon \cC^B\to \cC^A$ the functor of pre-composition with $f$, by $f_!$ its left adjoint and by $f_*$ its right adjoint (when they exist).
   
    \item We regard abelian groups as discrete connective spectra. In particular, we omit the standard notation $A\mapsto HA$ for the fully faithful embedding $\Ab \subseteq\Sp^\cn$. 
    
    \item We denote the tensor product of modules over a commutative ring spectrum $R$ simply by $\otimes$, or, if $R$ is not clear from the context, by $\otimes_R$. 
\end{itemize}
\subsection{Acknowledgements}

First, I would like to thank Thomas Nikolaus for the discussions that led me to think about this problem. I want to thank also Robert Burklund, Dustin Clausen, Jeremy Hahn, Markus Land, Akhil Mathew, Tyler Lawson, Maxime Ramzi, Tomer M. Schlank, Lior Yanovski, and Allen Yuan for valuable conversations, comments, and ideas related to this project, and to Shay Keidar, Tomer M. Schlank, and Allen Yuan for their useful comments on an earlier draft. The author is partially supported by the Danish National Research Foundation
through the Copenhagen Centre for Geometry and Topology (DNRF151).

\section{The Frobenius and Canonical Maps}

A commutative ring spectrum $R$ has two remarkable maps to the Tate construction $R^{tC_p}$; the canonical map 
\[
\can^p \colon R\to R^{tC_p},
\]
analogous to the reduction mod $p$ map
$
R\to R/p
$
in ordinary commutative algebra, 
and the Tate-valued Frobenius map 
\[
\fr^p\colon R\to R^{tC_p},
\]
analogous to the $p$-power map $R\to R/p$, which was constructed in \cite{nikolaus2018topological}.


Our analysis of the strict Picard spectrum of commutative ring spectra depends crucially on the nature of the extension of scalars functors from perfect $R$-modules to perfect $R^{tC_p}$-modules: 
\[
\fr^p_*,\can^p_*\colon \Perf_{R}\to \Perf_{R^{tC_p}}
\] 
In this section, we recall the construction of the maps $\can^p$ (in \Cref{subsec:The Canonical Map}) and $\fr^p$ (in \Cref{subsec:The Tate-Valued Frobenius}), and provide the required interpretation of the corresponding extension of scalar functors.

\subsection{The Canonical Map} 
\label{subsec:The Canonical Map}
The canonical map is a natural transformation $X\to X^{tC_p}$ defined for every object of a presentable stable $\infty$-category $X$. This map is a homotopy-theoretic upgrade of the reduction mod $p$ map $A\to A/p$ for abelian groups. We now recall the Tate-construction, $(-)^{tC_p}$, which is the target of the canonical map, define the canonical map, and recall the multiplicative properties of these to constructions. Finally, for a commutative ring spectrum $R$, we identify the functor of extension of scalars of perfect $R$-modules along the canonical map in terms of Tate construction.
\subsubsection{Tate Construction}
\begin{defn}[cf. {{\cite[Definition I.1.13]{nikolaus2018topological}}}]
For a finite group $G$ and a presentable stable $\infty$-category $\cC$, we denote by 
\[
(-)^{tG}\colon \cC^{BG}\to \cC
\]
the \tdef{Tate construction} of $G$, given by the cofiber of the norm map 
 \[
 \Nm_G\colon (-)_{hG}\to (-)^{hG}. 
 \]
 \end{defn}

The following characterization of $(-)^{tG}$ is shown in \cite[Theorem I.3.1]{nikolaus2018topological} for $\cC=\Sp$, but the proof works for every presentably symmetric monoidal stable $\infty$-category. 
\begin{prop} \label{Tate_lax}
The functor $(-)^{tG}\colon \cC^{BG} \to \cC$ is the initial exact, lax symmetric monoidal functor under $(-)^{hG}$ which carries induced local systems to $0$.
\end{prop}
Here, by ''induced," we mean local systems in the image of the functor $e_!\colon \cC \to \cC^{BG}$, given by the left Kan extension along the inclusion of the basepoint. 
\begin{rem}
Since $(-)^{tC_p}$ is lax symmetric monoidal, it promotes to a functor $\cC^{BG} \to \Mod_{\one^{tG}}(\cC)$. We will abuse notation and denote this functor again by $(-)^{tG}$. 
\end{rem}

\subsubsection{The canonical map}

By construction, we have natural maps $X^{hG}\to X^{tG}$ for every $X\in \cC^{BG}$. If $X$ is endowed with a trivial $G$-action, we also have a unit map $u\colon X\to X^{hG}$, given informally by embedding the constant maps $BG\to X$ into all the maps of this type. Composing them, we obtain the following:
\begin{defn}
Let $\cC \in \calg(\Prl^\st)$ and let $X\in \cC$, endowed with the trivial $G$-action. We define the \tdef{canonical map} 
\[
\mdef{\can_X^G}\colon X\to X^{tG} 
\]
to be the composite
\[
X\oto{u}X^{hG} \to X^{tG}.
\]  
 \end{defn}
 Thus, the maps $\can_X^G$ assemble to a lax symmetric monoidal natural transformation $\can^G\colon \Id_\cC \to (-)^{tG}$, where again, the target is the formation of Tate construction on \emph{constant} local systems over $BG$. The main case of interest for us is $G=C_p$, and we denote $\can^{C_p} := \can^p$ in this case. 
 
 Since the functor $(-)^{tG}$ is lax symmetric monoidal, it caries the unit to a commutative algebra $\one^{tG}$ in $\cC$. Thus, $(-)^{tG}$ promotes to a functor into $\one^{tG}$-modules, and by the free forgetful adjunction 
 \[
 \Mod_{\one^{tG}}(\cC) \adj \cC,
 \] 
 $\can^G$ corresponds to a  $\one^{tG}$-linear map
 \[
 \tilde{\can}^G\colon X\otimes \one^{tG} \to X^{tG}.  
 \]
 This map is an isomorphism in the case of perfect modules over commutative ring spectra. 
\begin{prop} \label{canonical_map_ext_scalars}
Let $R$ be a commutative ring spectrum. The map 
\[
\tilde{\can}^G \colon X\otimes R^{tG}\to X^{tG}
\] 
is an isomorphism for every perfect $R$-module $X\in \Perf_R$. Consequently, the functor 
$(-)^{tG}\colon \Perf_R \to \Mod_{R^{tG}}$ lands in $\Perf_{R^{tG}}$ and given by extension of scalars along the map $\can^G_R\colon R\to R^{tG}$. 
\end{prop} 
 
\begin{proof}
The source and target of $\tilde{\can}^G$ are exact functors of $X$, and since $\Perf_R$ is generated from $R$ under finite colimits, desuspensions, and retracts, the claim for all perfect $R$-modules follows from the special case $X=R$, which is clear.  
\end{proof}
\subsection{The Tate-Valued Frobenius}
\label{subsec:The Tate-Valued Frobenius}
 The Tate-valued Frobenius is a natural map $\fr^p\colon R \to R^{tC_p}$ defined for commutative ring spectra. We now define the map $\fr^p$ and identify, as for the canonical map, the effect of extending scalars along $\fr^p$.
\subsubsection{Equivariant Powers} 
For the definition of the Tate-valued Frobenius, we need the notion of equivariant powers. 
\begin{defn}
Let $\cC$ be a symmetric monoidal $\infty$-category, so that we have a $C_p$-equivariant $p$-fold tensor product functor $\cC^p \to \cC$. The symmetric monoidal  composition
\[
\cC \oto{\Delta} \cC^{p} \oto{\otimes} \cC,
\]
where $\Delta$ is the $p$-fold diagonal,
is $C_p$-equivariant for the trivial action on the source and target, and hence classifies a symmetric monoidal functor 
\[
\mdef{\Pow^p_\cC} \colon \cC \to \cC^{BC_p}
\]
that we refer to as the $p$-th \tdef{equivariant power} functor. When $\cC$ is clear from the context, we omit it from the notation and write $\Pow^p:= \Pow^p_\cC$. For a commutative ring spectrum $R$, we denote $\Pow^p_R:= \Pow^p_{\Mod_R}$. 
\end{defn}
Informally, $\Pow^p_\cC(X) = X^{\otimes p}$, endowed with the $C_p$-action that permutes the tensor factors. 
\begin{rem}
There is another perspective on $\Pow^p_\cC$ that will be useful later. Namely, the $\infty$-category $\calg(\cat_\infty)$ of symmetric monoidal $\infty$-categories is semiadditive. Consequently, for every $\cC \in \calg(\cat_\infty)$ and map $f\colon A\to B$ of spaces with finite fibers, we have a transfer, or ''integration", map $\int_f \colon \cC^A \to \cC^B$, see, e.g. \cite[Definition 2.1.11]{TeleAmbi}. Unwinding the definitions, the functor $\Pow^p_\cC$ identifies with the transfer $\int_e\colon \cC \to \cC^{BC_p}$ along the base-point inclusion $e\colon \pt \to BC_p$ in the semiadditive $\infty$-category $\calg(\cat_\infty)$.   
\end{rem}  

In the case of stable $\infty$-categories, the functor $\Pow^p$ has controlled behavior regarding cofiber sequences. However, before we formulate and prove it, let us first give a brief digression on filtered and graded objects. 
For $\cC \in \cat^\st$, 
let $\Fil(\cC)$ be the $\infty$-category of ($\ZZ$-)filtered objects in $\cC$, i.e., sequences of the form 
\[
\dots \to X_{-1} \to X_0 \to X_1 \to \dots
\] 
of objects of $\cC$,
and let $\Gr(\cC)$ be the $\infty$-category of $\ZZ$-graded objects, see, e.g., \cite[\S 3.1]{lurie2015rotation}\footnote{The presentation in this paper is for spectra, but the construction works for general stable $\infty$-category.}. When $\cC$ is presentably symmetric monoidal, both these $\infty$-categories admit  symmetric monoidal structures given by Day convolution, and the associated graded functor $\gr\colon \Fil(\cC) \to \Gr(\cC)$ is symmetric monoidal (\cite[Proposition 3.2.1]{lurie2015rotation}). 
For $X\in \cC$, we denote by $X(k)$ the object $X$ of $\Gr(\cC)$, regaded as sitting in degree $k$. 

\begin{prop}
\label{equiv_power_filtration}
Let $\cC \in \calg(\cat^\st)$, and let $X\to Y \to Z$ be a cofiber sequence in $\cC$. Then, 
$
\Pow^p(Y)
$
admits a canonical finite filtration whose associated graded pieces consist of: 
\begin{itemize}
    \item $\Pow^p(X)$. 
    \item $e_! (X^{\otimes |I|}\otimes Y^{\otimes p-|I|})$ for every orbit of $C_p$ in the set of non-empty proper subset $I\subseteq \{1,\dots,p\}$. 
    \item $\Pow^p(Z)$.
\end{itemize}
\end{prop}

\begin{proof}
First, the general stable case follows from the case where $\cC$ is presentable by replacing $\cC$ with $\Ind(\cC)$, and using the fully faithful symmetric monoidal embedding $\cC \into \Ind(\cC)$. 
The case where the cofiber sequence splits, so that $Y\simeq X\oplus Z$, is explained e.g. in \cite[Lemma 3.4.9]{TeleAmbi}, and moreover in this case the filtration on $\Pow^p(Y)$ splits into a direct sum decomposition 
\begin{equation} \label{equiv_pow_decomposition}
\Pow^p(X\oplus Z)\simeq \Pow^p(X)\oplus \Pow^p(Z) \oplus \bigoplus_I e_!(X^{\otimes |I|}\otimes Z^{\otimes p-|I|}).  
\end{equation}
For the general case, consider the filtered object
\[
F : \dots=0=0\to  X\to Y = Y = \dots \qin \Fil(\cC). 
\] 
in which the unspecified map $X\to Y$ has cofiber $Z$. 
Note that $\Pow^p_{\Fil(\cC)}(F)$ is a filtration on the object $\Pow^p(Y)$. Moreover, since
\[
\gr(F) \simeq X(0)\oplus Z(1),
\]  
and $\gr$ is symmetric monoidal, we have, by the case of a split extension applied to $\Gr(\cC)$:
\[
\gr(\Pow^p_{\Fil(\cC)}(F))\simeq \Pow^p_{\Gr(\cC)}(\gr(F))\simeq \Pow^p(X)(0)\oplus \Pow^p(Z)(p) \oplus \bigoplus_I e_!(X^{\otimes |I|}\otimes Z^{\otimes p-|I|})(|I|).  
\]
Thus, $\Pow^p_{\Fil(\cC)}(F)$ is a filtration on $Y$ with the required associated graded pieces. 
\end{proof}
  
\subsubsection{The Frobenius Twist Functor and the Tate Diagonal}  
Informally, \Cref{equiv_power_filtration} shows that $\Pow^p$ is exact modulo induced local systems. Thus, it becomes exact after applying the Tate construction. This motivates the definition of the following functor. 

\begin{defn}
Let $\cC\in \calg(\Prl^\st)$. We define the functor 
\[
\mdef{\Fr^p_\cC} \colon \cC \to \Mod_{\one^{tC_p}}(\cC)
\]
by the composition 
\[
\cC \oto{\Pow^p_\cC} \cC^{BC_p}\oto{(-)^{tC_p}} \Mod_{\one^{tC_p}}(\cC). 
\] We refer to $\Fr^p_\cC$ as the \tdef{Frobenius twist functor}\footnote{The functor $\Fr_\Sp^p$ is referred to as the \emph{Topological Singer construction} in \cite{rognes2010topological}.} of $\cC$.  
\end{defn}

Note that $\Fr^p_\cC$, being a composition of symmetric monoidal and lax symmetric monoidal functors, is itself lax symmetric monoidal.
As indicated before its definition, the functor $\Fr^p_\cC$ has the following  key additional feature:
\begin{prop} [cf. {{\cite[Proposition III.1.1]{nikolaus2018topological}}}] \label{Frobenius_Exact}
Let $\cC \in \calg(\Prl^\st)$. The lax symmetric monoidal functor $\Fr^p_\cC$ is exact. 
\end{prop}
 
\begin{proof}
Since $\Fr^p_\cC(0)\simeq 0$, it suffices to check that it preserves cofiber sequences. Let $X\oto{f} Y\oto{g} Z$ be a cofiber sequence. By \Cref{equiv_power_filtration} and the fact that $(-)^{tC_p}$ vanishes on induced local systems, the object $\Fr^p_\cC(Y)$ admits a filtration with associated graded pieces $\Fr^p_\cC(X)$ and $\Fr^p_\cC(Z)$. By inspecting its construction, this filtration is given (up to ''spacing") by 
\[
\dots= 0 = 0 \to \Fr^p_\cC(X) \oto{\Fr^p_\cC(f)} \Fr^p_\cC(Y) = \Fr^p_\cC(Y) = \dots
\]
and hence correspond to a cofiber sequence
\[
\Fr^p_\cC(X)\oto{\Fr^p_\cC(f)} \Fr^p_\cC(Y)\oto{\Fr^p_\cC(g)} \Fr^p_\cC(Z).
\]
Since this sequence is obtained from $X\oto{f} Y \oto{g} Z$ by applying $\Fr_\cC^p$, we deduce that $\Fr_\cC^p$ carries cofiber sequences to cofiber sequences. 
\end{proof}

Specializing in the case $\cC = \Sp$, we now define the Tate diagonal map. 

\begin{defn}[cf. {{\cite[Definition III.1.4]{nikolaus2018topological}}}]
For $\cC = \Sp$, every exact and lax symmetric monoidal functor $\Sp \to \Sp$ receives a unique lax symmetric monoidal natural transformation from $\Id_{\Sp}$. In particular, we obtain such a natural transformation
\[
\mdef{\Delta} \colon \Id_\Sp \to \Fr^p_\Sp,
\]
which we refer to as the \tdef{Tate diagonal}. 
\end{defn}

\subsubsection{The Tate-Valued Frobenius} 
Let $R$ be a commutative ring spectrum. The $p$-fold multiplication map $R^{\otimes p}\to R$ is $C_p$-equivariant and hence can be viewed as a map $ \Pow^p(R)\to R$. Here, the target is endowed with the trivial $C_p$-action. Applying the Tate construction, we obtain a map 
\[
m_R\colon \Fr^p_\Sp(R)\to R^{tC_p}. 
\]

\begin{defn}[cf. {{\cite[Definition IV.1.1]{nikolaus2018topological}}}]
Let $R$ be a commutative ring spectrum.  We define the \tdef{Tate-valued Frobenius} map
\[
\mdef{\fr_R^p}\colon R \to R^{tC_p} 
\]
by the composition 
\[
R \oto{\Delta} \Fr^p_\Sp(R) \oto{m_R}R^{tC_p}. 
\]
\end{defn}

As for the canonical map, we wish to identify the extension of scalars functor $\fr_*^p\colon \Perf_R \to \Perf_{R^{tC_p}}$ from perfect $R$-modules to perfect $R^{tC_p}$-modules. 

Let $\Pow^p_R\colon \Mod_R\to \Mod_R^{BC_p}$ be the equivariant power functor for the $\infty$-category $\Mod_R$. 
The lax monoidality of the forgetful functor $\Mod_R\to \Sp$ supplies a natural transformation 
\[
\Pow^p_\Sp(X)\to \Pow^p_R(X), \quad X\in \Mod_R,
\] 
where we regard the target as valued in $\Sp$ via the forgetful functor. Applying the Tate construction, we obtain a natural transformation 
\[
r\colon \Fr^p_\Sp(X)\to \Fr^p_R(X). 
\]
As in the discussion above \Cref{canonical_map_ext_scalars},
the source of $r$ admits a natural structure of an $\Fr^p_\Sp(R)$-module while the target that of a $\Fr^p_R(R) \simeq R^{tC_p}$-module. The map $r$ intertwine these structures along the ring map $r_R\colon \Fr^p_\Sp(R)\to R^{tC_p}$ and hence it classifies a natural map 
\[
\tilde{r}\colon \Fr^p_{\Sp}(X)\otimes_{\Fr^p_\Sp(R)}R^{tC_p}\to \Fr^p_R(X) \qin \Mod_{R^{tC_p}}.
\]
Similarly, the Tate diagonal map 
\[
\Delta \colon X\to \Fr^p_\Sp(X) 
\]
classifies a map 
\[
\tilde{\Delta} \colon X\otimes_{R} \Fr^p_{\Sp}(R) \to \Fr^p_\Sp(X). 
\]

\begin{prop} \label{extension_of_scalars_diagonal_mult}
Let $R$ be a commutative ring spectrum. The two natural maps $\tilde{r}$ and $\tilde{\Delta}$  introduced above are isomorphisms when evaluated at perfect $R$-modules.  
\end{prop}

\begin{proof}
As in the proof of \Cref{canonical_map_ext_scalars}, in both cases, the source and target of the natural transformations are exact functors of the input module $X$, and therefore the claim for a general perfect module reduces to the case $X=R$, which is clear.   
\end{proof}

We can now identify the Frobenius twist functor for perfect $R$-modules.

\begin{cor}\label{Frobemious_ext_scalars}
Let $R$ be a commutative ring spectrum. The restriction of the functor $\Fr^p_R\colon \Mod_R \to \Mod_{R^{tC_p}}$ to perfect $R$-modules identifies, as a lax symmetric monoidal functor, with the functor of extension of scalars along the commutative ring spectrum map $\fr^p\colon R\to R^{tC_p}$. In particular, it is symmetric monoidal and lands in perfect $R^{tC_p}$-modules.
\end{cor}

\begin{proof}
By \Cref{extension_of_scalars_diagonal_mult}, we have, naturally (and lax symemtric monoidally) in $X\in \Perf_R$: 
\[
\Fr^p_R(X)\simeq \Fr^p_\Sp(X)\otimes_{\Fr^p_\Sp(R)} R^{tC_p} \simeq X\otimes_{R}\Fr^p_{\Sp}(R) \otimes_{\Fr^p_\Sp(R)} R^{tC_p} \simeq X\otimes_R R^{tC_p}.
\]
Here, $\Fr^p_\Sp(R)$ is regarded as an $R$-algebra via the Tate-diagonal $\Delta$ and $R^{tC_p}$ as an $\Fr^p_{\Sp}(R)$-algenra via the multiplication map $m_R$. Hence, $\Fr^p_R$ is given by extension of scalars along the composition $\fr^p_R= m_R \circ \Delta$. 
\end{proof}

\section{Strict Picard Spectra} 
 In this section, we discuss the main object of study in this paper: strict Picard spectra. We start in \Cref{subsec:The Picard Spectrum} by reviewing the construction and basic properties of the Picard spectrum of symmetric monoidal $\infty$-categories and relate it to the classical Picard group in the case of modules over a ring. Then, in \Cref{subsec:Strict Elements} we define the connective $\ZZ$-module of ''strict elements" in any connective spectrum $X$ as the connective spectrum of maps $\ZZ\to X$. We compute the strict elements in $1$-truncated connective spectra, and show that the formation of strict elements trivializes the transfer maps associated with arbitrary finite groups. 
 
 Finally, in \Cref{subsec:The Strict Picard Spectrum} we introduce the strict Picard spectrum of a symmetric monoidal $\infty$-category and show that the equivariant $\otimes$-power functor of the category degenerates to a usual $p$-th tensor power map when restricted to the strict Picard spectrum.  
\subsection{The Picard Spectrum}
\label{subsec:The Picard Spectrum}
Let $\CMon(\Spc)$ be the $\infty$-category of commutative monoids.
If $\cC$ is a symmetric monoidal $\infty$-category, the space of objects $\cC^\simeq$ admits a commutative monoid structure via the tensor product, so that $\cC^\simeq\in \CMon(\Spc)$.   
Recall that the $\infty$-category $\Sp^\cn$ of connective spectra is equivalent to the full subcategory of $\CMon(\Spc)$ spanned by the \emph{grouplike} commutative monoids, i.e., those commutative monoids $M$ for which the monoid $\pi_0(M)$ is an abelian group. 
The fully faithful embedding 
\[
\Sp^\cn\into \CMon(\Spc)
\] 
admits a right adjoint, given by the formation of the \emph{maximal grouplike submonoid}. 
\begin{defn}
Let $\cC$ be a symmetric monoidal $\infty$-category. 
\begin{itemize}
    \item We denote by $\mdef{\pic(\cC)}\in \Sp^\cn$ the maximal grouplike submonoid of $\cC^\simeq$, and refer to $\pic(\cC)$ as the \tdef{Picard spectrum} of $\cC$.
    \item We denote by $\mdef{\Pic(\cC)}:=\pi_0\pic(\cC)$ the \tdef{Picard group} of $\cC$. \footnote{It is also common in the literature to denote by $\Pic(\cC)$ the $\infty$-loop space associated with the connective spectrum $\pic(\cC)$. We implicitly identify connective spectra with their corresponding $\infty$-loop spaces and reserve the notation $\Pic(\cC)$ for the group of components.}
    \item For a commutative ring spectrum $R$, we denote $\pic(R):= \pic(\Mod_R)$. 
\end{itemize} 
 
\end{defn}
Hence, $\pic(\cC)$ is the connective spectrum of  $\otimes$-invertible objects in $\cC$, with (commutative) multiplication given by the tensor product of $\cC$. 
In particular, we have a natural isomorphism 
\begin{equation} \label{eq:loops_pic_GL1}
    \Omega \pic(\cC) \simeq \one_\cC^\times,
\end{equation}
where $\one_\cC^\times$ is the spectrum of \tdef{units} of $\one_\cC$, see, e.g. \cite[\S 3.1]{carmeli2021chromatic}. In particular,
\begin{equation}\label{eq:pi_1_pic}
\pi_1\pic(\cC) \simeq (\pi_0\one_\cC)^\times. 
\end{equation}
Consequently, the element $\eta\in \pi_1\Sph$ gives a map $\Pic(\cC) \to (\pi_0\one_\cC)^\times$ for every $\cC \in \calg(\cat_\infty)$. 

\begin{prop}[cf. {{\cite[Proposition 3.20]{carmeli2021chromatic}}}]
\label{dim_eta}
Let $\cC$ be a symmetric monoidal $\infty$-category.
The map $\eta\colon \Pic(\cC)\to (\pi_0\one_\cC)^\times$ is given by 
\[
\eta\cdot[\mathcal{L}] = \dim(\mathcal{L}).
\]
Here, $\dim$ denote the symmetric monoidal dimension (a.k.a. ''Euler Characteristic"). 
\end{prop}

Given this result, the following subgroup of $\Pic(\cC)$ will play a role in our discussion of the strict Picard spectrum: 

\begin{defn}
Let $\cC$ be a symmetric monoidal $\infty$-category. We let $\Pic^0(\cC)$ be the subgroup of $\Pic(\cC)$ consisting of elements $[\mathcal{L}]\in \Pic(\cC)$ for which $\dim(\mathcal{L})=1$. 
\end{defn}
\subsubsection{Module Categories}
Already for an ordinary ring  $R\in \calg(\Ab)$ the spectrum $\pic(R)$ is interesting. It is closely related to the following more classical invariants. Let $\Mod_R^\heartsuit$ be the abelian category of ordinary $R$-modules.
\begin{defn}
For a ring $R$, we let 
\[
\mdef{\Cl(R)}:= \Pic(\Mod_R^\heartsuit). 
\]
We refer to $\Cl(R)$ as the \tdef{classical Picard group} of $R$.\footnote{This group is more commonly called just the Picard group of $R$ in the algebraic geometry literature. We use the term ''classical Picard group" to avoid possible confusion with the Picard spectrum of the category of $R$-module spectra.}
\end{defn}
Thus, $\Cl(R)$  consists of isomorphism classes of invertible discrete $R$-modules.

\begin{defn}
For a ring $R$, let $C^0(R;\ZZ)$ be the group of locally constant functions $\Spec(R)\to \ZZ$. 
\end{defn}
We can expand an element $f\in C^0(R;\ZZ)$ as a sum 
  \[f=
  \sum_{i=1}^k n_i \cdot \delta_{\varepsilon_i},
  \]
  where $\{n_1,\dots,n_k\}= \mathrm{Im}(f)$,  $\varepsilon_i$ is an idempotent with support $f^{-1}(n_i)$, and $\delta_{\varepsilon_i}$ is the function which is $1$ on the support of $\varepsilon_i$ and $0$ outside of it.      
\begin{prop} 
\label{Pic_ordinary_ring}
Let $R$ be a commutative ring, viewed as a discrete commutative ring spectrum. Then 
\[
\Pic(R)\simeq \Cl(R)\oplus C^0(R;\ZZ),
\]
where the isomorphism is given by
\[
(\sum_i n_i\cdot \delta_{\varepsilon_i},\mathcal{L})\mapsto \prod_{i} \Sigma^{n_i} \mathcal{L}[\varepsilon_i^{-1}] \qin \Pic(R).  
\]
Moreover, via this identification, the multiplication by $\eta\in \pi_1\Sph$ on $\Pic(R)$ is given by: 
\[
\eta \cdot (\sum_i n_i\cdot \delta_{\varepsilon_i},\mathcal{L}) = \sum_i(-1)^{n_i} \varepsilon_i \qin R^\times.
\] 

\end{prop}
\begin{proof}
The identification of $\Pic(R)$ is given in \cite[Theorem 3.5]{fausk2003picard}. 
It remains to identify the multiplication by $\eta$. Fix an element $\prod_i \Sigma^{n_i}\mathcal{L}[\varepsilon_i^{-1}]\in \Pic(R)$. Then $R\simeq \prod_i R[\varepsilon_i^{-1}]$ and hence $\pic(R)\simeq \prod_i \pic(R[\varepsilon_i^{-1}])$. Restricting to the $i$-th coordinate and replacing $R$ by $R[\varepsilon_i^{-1}]$, we reduce the computation to that of $\eta\cdot [\Sigma^nL]$ for $\mathcal{L}\in \Cl(R)$. The result now follows from \Cref{dim_eta}. 
Indeed, since objects of $\Cl(R)$ are of dimension $1$, we get 
\[
\dim(\Sigma^n\mathcal{L})=(-1)^n \dim(\mathcal{L}) = (-1)^n
\] 
in this case.  
\end{proof}

\subsection{Strict Elements}
\label{subsec:Strict Elements}
Let $\Mod_\ZZ^\cn$ be the $\infty$-category of connective $\ZZ$-module spectra. The forgetful functor 
\[
\Mod_\ZZ^\cn \to \Sp^\cn
\]
admits a left adjoint, $\ZZ\otimes (-)\colon \Sp^\cn \to \Mod_\ZZ^\cn$, and a right adjoint, $\hom_{\Sp^\cn}(\ZZ,-)\colon \Sp^\cn \to \Mod_\ZZ^\cn$. We shall mostly consider the right adjoint. 
\begin{defn}
\label{def:strict_elements}
We denote the functor $\hom_{\Sp^\cn}(\ZZ,-)$ by
\[
X\mapsto \mdef{\zlin{X}}.
\]
We shell refer to $\zlin{X}$ as the spectrum of \tdef{strict elements} in $X$.
\end{defn}

In general, computing the homotopy groups of $X_\ZZ$ is a challenging task, closely related to resolving $\ZZ$ by free $\Sph$-modules. However, for $X$ sufficiently truncated, $X_\ZZ$ can be computed explicitly. We shall need the following very special case.
\begin{prop}[cf. {{\cite[Proposition 3.23]{carmeli2021chromatic}}}]
\label{strict_elements_1_truncated}
Let $\eta\in \pi_1\Sph$ be the Hopf element and let $X$ be a connective, $1$-truncated spectrum. Then
\[
\pi_i\zlin{X}= 
\begin{cases}
\Ker(\pi_0X \oto{\eta} \pi_1X) & i=0 \\
\pi_1X & i=1 \\ 
0      & i\ge 2.
\end{cases}
\]
Equivalently, 
\[
\zlin{X}\simeq \Ker(\pi_0X\oto{\eta}\pi_1X) \oplus \Sigma \pi_1X. 
\]
\end{prop}

\begin{proof}
Let $\Sph/\eta$ be the cofiber of the morphism $\Sigma \Sph\oto{\eta} \Sph$. Since
 $\tau_{\le 1}\Sph/\eta \simeq \ZZ$ and $X$ is $1$-truncated, we deduce that 
 \[
 \zlin{X}=\hom_{\Sp^\cn}(\ZZ,X)\simeq \hom_{\Sp^\cn}(\Sph/\eta,X). 
 \]
Applying $\hom_{\Sp}(-,X)$ to the cofiber sequence 
\[
 \xymatrix{
 \Sigma \Sph \ar^\eta[r] & \Sph \ar[r] & \Sph/\eta
 }
 \]
We obtain a cofiber sequence  
 \[
 \xymatrix{
 X_\ZZ \ar[r] & X \ar^{\eta}[r] & \Omega X
 }
 \]
 in $\Sp$. 
 The result now follows from the associated long exact sequence of homotopy groups (note that $\hom_{\Sp^\cn} = \tau_{\ge 0}\hom_\Sp$). The ''equivalently" part now follows from the fact that $\ZZ$-module spectra decompose into a product of  Eilenberg-MacLane spectra.
\end{proof}
\subsubsection{Transfer Maps}
Let $G$ be a finite group. We have two maps 
\[
\pi\colon BG\to \pt
\]
and 
\[
e\colon \pt\to BG,
\]
which contract $BG$ to a point and include its basepoint respectively. 
As a result, for an $\infty$-category $\cC$ which admits homotopy fixed points for finite group actions, and for $X\in \cC$, we have maps 
\[
\pi^*\colon X\to X^{BG}  
\]
and
\[
e^*\colon X^{BG}\to X, 
\]
given, informally, by the formation of constant maps to $X$ and evaluation at the basepoint.

In fact, since $e$ has finite discrete fibers, if $\cC$ is semiadditive, there is also a natural transfer map
\[
\push_e \colon X\to X^{BG},
\]
(see, e.g., \cite[Notation 3.1.8]{TeleAmbi} for a more general construction).
Transfer maps are particularly simple in the $\infty$-category of $\ZZ$-module spectra. 
\begin{prop} \label{triv_tras_ZZ}
Let $G$ be a finite group and let $\pi\colon BG \to \pt$ and $e\colon \pt \to BG$ be the projection to the point and inclusion of the base point, respectively. The diagram 
\begin{equation}
\label{eq:triv_tras_ZZ_triangle}
     \xymatrix{
\Id_{\Mod_\ZZ} \ar^{\push_e}[r]\ar_{|G|}[d] & \Id_{\Mod_\ZZ}^{BG} \\ 
\Id_{\Mod_\ZZ} \ar_{\pi^*}[ru]  & 
}
\end{equation}
commutes in $\End(\Mod_\ZZ)$. In particular, for every $\ZZ$-module spectrum $X$, the triangle  
\[
\xymatrix{
X\ar_{|G|}[d] \ar^{\int_e}[r] & X^{BG} \\ 
X \ar_{\pi^*}[ru] 
}
\]
commutes.
\end{prop}

\begin{proof}
The $\infty$-category $\Mod_\ZZ$ is closed symmetric monoidal, so we have an (internal) co-Yoneda functor $\hom\colon \Mod_\ZZ^\op \to \End(\Mod_\ZZ)$, taking $X\in \Mod_\ZZ^\op$ to the functor $Y\mapsto \hom(X,Y)$. The vertices and edges of the claimed diagram (\ref{eq:triv_tras_ZZ_triangle}) are in the image of this functor. More precisely, we have: 
\begin{itemize}
    \item $\Id_{\Mod_\ZZ} \simeq \hom(\ZZ,-)$. 
    \item $\Id_{\Mod_\ZZ^{BG}} \simeq \hom(\ZZ[BG],-)$. 
    \item The map $|G|\colon \Id_{{\Mod_\ZZ}} \to \Id_{{\Mod_\ZZ}}$ is obtained from $|G|\colon \ZZ \to \ZZ$ by applying $\hom$. 
    \item The map $\pi^*\colon \Id_{{\Mod_\ZZ}}\to \Id_{{\Mod_\ZZ}}^{BG}$ is obtained from the map $\pi_*\colon \ZZ[BG]\to \ZZ$ induced from $\pi$ on $\ZZ$-valued chains by applying $\hom$. 
    \item The map $\push_e$ is obtained from the transfer map in $\ZZ$-valued chains $\tr_G \colon \ZZ[BG] \to \ZZ$ by applying $\hom$.   
\end{itemize}

By the functoriality of $\hom$, it remains to show that the triangle
\[
\xymatrix{
\ZZ[BG] \ar^\pi[r] \ar_{\tr_G}[rd] & \ZZ\ar^{|G|}[d] \\ 
          & \ZZ 
}
\]
commutes. 

Since $\ZZ$ is discrete and the map $e_*\colon \ZZ\to \ZZ[BG]$ induces an isomorphism on $\pi_0$, it is enough to check the commutativity of this triangle after pre-composing with $e_*$. This, in turn, follows from the facts that $\tr_G\circ e_* = |G|$ and $\pi_*e_*=1$. 
\end{proof}

As a result, we see that the functor $(-)_\ZZ$ indeed ''strictifies" the transfer maps. Namely,
\begin{cor}
\label{transfer_strict_trivial}
With the same settings as in \Cref{triv_tras_ZZ}, for $X\in \Sp^\cn$, the two maps 
\[
X\oto{\push_e} X^{BG} 
\]
and
\[
X\oto{p} X \oto{\pi^*} X^{BG}
\]
induce homotopic maps 
\[
\zlin{X}\to \zlin{(X^{BG})}. 
\]
\end{cor}

\begin{proof}
The functor $\zlin{(-)}\colon \Sp^\cn\to \Mod_\ZZ^\cn$ is limit preserving, and hence:
\begin{itemize}
    \item It takes the transfer map $\push_e\colon X\to X^{BG}$ to the transfer map 
\[
\push_e\colon \zlin{X}\to \zlin{(X^{BG})}\simeq (\zlin{X})^{BG},
\]
(see, e.g., \cite[Corollary 3.2.7]{TeleAmbi}).
\item It takes the pullback map $\pi^*\colon X\to X^{BG}$ to the pullback map 
\[
\pi^*\colon \zlin{X}\to \zlin{(X^{BG})}\simeq (\zlin{X})^{BG},
\]
since post-composition with $(-)_\ZZ$ and pre-composition with $\pi$ commute. 
\item It intertwines the multiplication by $p$ maps since it is an additive functor. 
\end{itemize}
Combining these three properties, it remains to show that the triangle 
\[
\xymatrix{
\zlin{X}\ar^p[d] \ar^{\push_e}[r] & (\zlin{X})^{BC_p} \\ \zlin{X}\ar_{\pi^*}[ru] & 
}
\]
commutes in $\Mod_\ZZ^\cn$. 
This, in turn, follows from \Cref{triv_tras_ZZ}.
\end{proof}

\subsection{The Strict Picard Spectrum}
\label{subsec:The Strict Picard Spectrum}
The strict Picard spectrum of $\cC$ is obtained from its Picard spectrum by strictification.

\begin{defn}
Let $\cC$ be a symmetric monoidal $\infty$-category. We denote  
\[
\spic(\cC):= \pic(\cC)_\ZZ, 
\]
and refer to $\spic(\cC)$ as the \tdef{strict Picard spectrum} of $\cC$. If $R$ is a commutative ring spectrum, we denote $\spic(R):= \spic(\Mod_R)$. We refer to points in $\spic(R)$ as \tdef{strictly invertible $R$-modules}.  
\end{defn}

\begin{rem} \label{spic_perf}
For a commutative ring spectrum $R$, 
every dualizable, and in particular invertible $R$-module is perfect. Hence, we have $\spic(R)\simeq \spic(\Perf_R)$. 
\end{rem}

\begin{rem} \label{spic_G_m}
Since $\zlin{(-)}$ is limit preserving, we obtain from \Cref{eq:loops_pic_GL1} that 
\[
\Omega \spic(\cC) \simeq \GG_m(\one_\cC):= \zlin{(\one_\cC^\times)}, 
\]
the spectrum of \tdef{strict units} of $\one_\cC$. 
\end{rem}

\subsubsection{Equivariant Powers and Transfer maps}

Restricting the equivariant power map $\Theta^p_\cC \colon \cC \to \cC^{BC_p}$ to the Picard spectra, we obtain a natural morphism 
\[
\pic(\Theta^p_\cC) \colon \pic(\cC) \to \pic(\cC^{BC_p}).
\] 
This map can be interpreted as a transfer map, this time in $\Sp^\cn$:
\begin{prop} \label{equiv_pow_transfer}
Let $e\colon \pt \to BC_p$ be the point inclusion, and let 
$\push_e \colon \Id_\Sp \to (-)^{BC_p}$ be the corresponding transfer map. Then, the map $\pic(\Theta^p_\cC)\colon \pic(\cC) \to \pic(\cC^{BC_p})\simeq \pic(\cC)^{BC_p}$ agrees with the $\pic(\cC)$-component of $\push_e$.
\end{prop}

\begin{proof}
As in \Cref{equiv_pow_transfer}, we can identify the functor 
\[
\Theta^p_\cC \colon \cC \to \cC^{BC_p} 
\]
with the transfer along $e$ in the $\infty$-category $\calg(\cat_\infty)$. The functor $\pic\colon \calg(\cat_\infty)\to \Sp^\cn$ is limit-preserving and hence intertwines the transfer maps (see, e.g., \cite[Corollary 3.2.7]{TeleAmbi}). This implies the result.
\end{proof}

Applying the functor $\zlin{(-)}$ to the morphism 
\[
\pic(\Theta^p_\cC) \colon \pic(\cC)\to \pic(\cC^{BC_p}) \qin \Sp^\cn,
\]
we get a morphism
\[
\spic(\Theta^p_\cC) \colon \spic(\cC) \to \spic(\cC^{BC_p}) \qin \Mod_\ZZ^\cn,
\]
functorial in $\cC\in \calg(\cat_\infty)$. 
In fact, the functor $\spic(-)$ trivializes $\Theta_\cC^p$, in the following sense: 
\begin{prop}\label{equiv_power_trivial_strict}
Let $\cC\in \calg(\cat_\infty)$ and let $\pi\colon BC_p \to \pt$ be the terminal map. The diagram 
\[
\xymatrix{
\spic(\cC) \ar^{p}[d] \ar^{\spic(\Theta^p_\cC)}[rr] & & \spic(\cC^{BC_p})     \\
\spic(\cC) \ar_{\pi^*}[rru] & &  
}
\]
commutes in $\Mod_\ZZ^{\cn}$. 
\end{prop}

\begin{proof}
By \Cref{equiv_pow_transfer}, we can identify the map $\pic(\Theta^p_\cC)$ with the transfer map along the base-point inclusion $e\colon \pt \to BC_p$. Hence, the result follows from \Cref{transfer_strict_trivial}.
 \end{proof}
 
\section{The Strict Picard Spectrum of Commutative Ring Spectra}
We now restrict our attention to strict Picard spectra of commutative ring spectra. In \Cref{subsec:The Frobenius Image of Strict Picard Elements} we use \Cref{equiv_power_trivial_strict} to show that the extension of scalars along $\fr^p$ and $\can^p$ differ by a $p$-th tensor power on strict Picard elements. Then, in \Cref{subsec:pefrect E_infty rings} we apply this to the strict Picard spectra of commutative ring spectra for which the map $\varphi^p$ is invertible. We show that for such a commutative ring spectrum $R$, multiplication by $p$ is invertible on $\spic(R)$, and it depends only on the $1$-truncation of $\pic(R)$. We also determine the strict Picard spectra of spherical Witt vectors over perfect rings of characteristic $p$. In particular, we determine the strict Picard spectrum of the $p$-complete spheres. 
Finally, in \Cref{subsec:trictly Invertible Spectra} we use the arithmetic fracture square and the determination of $\spic(\Sph_p)$ to compute the strict Picard spectrum of the sphere spectrum.

\subsection{The Frobenius Image of Strictly Invertible Modules}
\label{subsec:The Frobenius Image of Strict Picard Elements}
Our goal in this section is to prove a relation between the maps $\fr^p\colon \spic(R)\to \spic(R^{tC_p})$ and $\can^p\colon \spic(R)\to \spic(R^{tC_p})$. The main technical difficulty is that their constructions involve the Tate construction for the group $C_p$, which is only lax symmetric monoidal on $(\Perf_R)^{BC_p}$, and hence does not induce maps on Picard spectra. To overcome this difficulty, we introduce the following alternative for $(\Perf_R)^{BC_p}$:

\begin{defn}
For a commutative ring spectrum $R$, define
\[
\mdef{\Perf_{R,C_p}}\subseteq (\Perf_R)^{BC_p}
\] 
to be the thick subcategory generated from the trivial local systems $\pi^*X$ and the induced local systems $e_!X$ for all $X\in \Perf_R$. Here, $\pi\colon BC_p\to \pt$ and $e\colon \pt \to BC_p$ are the terminal map and the base-point inclusion respectively.  
\end{defn}
\begin{rem} \label{rem:generators_equiv_gen}
Note that, since the thick subcategory of $\Perf_R$ generated from $R$ is all of $\Perf_R$, the stable $\infty$-category $\Perf_{R,C_p}$ is generated from $\pi^*R$ and $e_!R$. 
\end{rem}

The $\infty$-category $\Perf_{R,C_p}$ has two key features that we shall use. First, there is no harm in replacing $(\Perf_R)^{BC_p}$ by $\Perf_{R,C_p}$. Namely:

\begin{lem} 
\label{equiv_power_constructible}
Let $R$ be a commutative ring spectrum. The functors $\pi^*,\Theta_R^p\colon \Perf_R\to (\Perf_R)^{BC_p}$ land in $\Perf_{R,C_p}$.
\end{lem}

\begin{proof}
For $\pi^*$ this is clear from the definition of $\Perf_{R,C_p}$. We prove the claim for $\Pow^p_R$.
It is a consequence of  \Cref{equiv_power_filtration} that the composition of $\Pow^p_R\colon \Perf_R \to (\Perf_R)^{BC_p}$ with the Verdier quotient 
$(\Perf_R)^{BC_p}\to (\Perf_R)^{BC_p}/\Perf_{R,C_p}$ is exact. Since it carries $R$ to the constant local system $\pi^*R$, which is sent to $0$ in the Verdier quotient, and since $\Perf_R$ is generated from $R$ under finite colimits, retracts, and desuspensions, we deduce that this composition is the zero functor. Namely, that $\Pow^p_R$ lands in $\Perf_{R,C_p}$. 
\end{proof}
The main advantage of $\Perf_{R,C_p}$ over $(\Perf_{R})^{BC_p}$ is the following:
\begin{lem} \label{tate_symm_mon_constructibles}
Let $R$ be a commutative ring spectrum. The lax symmetric monoidal functor \[
(-)^{tC_p}\colon \Perf_{R,C_p} \to \Mod_{R^{tC_p}}
\] 
is symmetric monoidal. 
\end{lem}

\begin{proof}

Since this functor is unital by design, we only have to show that the lax monoidality map 
\[
\nu_{X,Y}\colon X^{tC_p}\otimes Y^{tC_p}\to (X\otimes Y)^{tC_p} 
\]
is an isomorphism for every $X,Y\in \Perf_{R,C_p}$. Since the source and target of $\nu$ are exact functors of $X$ and $Y$ separately, and in view of \Cref{rem:generators_equiv_gen}, it suffices to prove this for $X=\pi^*R$ or $X= e_!R$. In the first case, the claim follows from the unitality of $(-)^{tC_p}$. In the second case, we have $(e_!R)^{tC_p} \otimes Y^{tC_p} = 0$ since $(-)^{tC_p}$ vanishes on the induced local system $e_!R$. On the other hand, by the ''projection formula" for $e_!$ we have 
\[
(e_!R)\otimes Y\simeq e_!e^*Y  
\] 
which is also induced, and hence $(e_!R\otimes Y)^{tC_p} = 0$ as well. We deduce that $\nu_{e_!R,Y}$ is a morphism between zero objects, hence an isomorphism. 
\end{proof}

We are ready to prove the main result of this subsection:
\begin{prop}\label{fundamental_identity}
Let $R$ be a commutative ring spectrum. The following triangle commutes in $\Mod_\ZZ^\cn$: 
\[
\xymatrix{
\spic(R) \ar^{p}[d] \ar^{\fr^p}[rr] & & \spic(R^{tC_p})     \\
\spic(R) \ar_{\can^p}[rru] & &  
}
\]
 In other words, for a strictly invertible $R$-module $X\colon \ZZ \to \pic(R)$, there is a  natural isomorphism 
\begin{equation}
\label{eq:foundamental_Indetity}
\fr^p_*(X)\simeq \can^p_*(X^{\otimes p}).
\end{equation}
\end{prop}

\begin{proof}
Recall that $\spic(R)= \spic(\Perf_R)$.
Now, we have a symmetric monoidal fully faithful embedding $\Perf_{R,C_p}\into (\Perf_R)^{BC_p}$. By \Cref{equiv_power_constructible}, the functors $\Pow_R^p$ and $\pi^*$ factor through $\Perf_{R,C_p}$. Consequently, using \Cref{Frobemious_ext_scalars}, we can decompose the functor $\fr^p_* \colon \Perf_R \to \Perf_{R^{tC_p}}$ as the composition of \emph{symmetric monoidal functors}
\[
\Perf_R \oto{\Pow^p_R} \Perf_{R,C_p} \oto{(-)^{tC_p}} \Perf_R. 
\]
Similarly, this time using \Cref{canonical_map_ext_scalars}, we can write the functor $\can^p\colon \Perf_R \to \Perf_{R,C_p}$ as the composition 
\[
\Perf_R \oto{\pi^*} \Perf_{R,C_p} \oto{(-)^{tC_p}} \Perf_R. 
\]

Finally, the fully faithfulness of the embedding $\Perf_{R,C_p}\into (\Perf_R)^{BC_p}$ implies that the map $\spic(\Perf_{R,C_p}) \to \spic((\Perf_R)^{BC_p})$ is an inclusion of connective $\ZZ$-modules, in the sense that it is injective on $\pi_0$ and induces a pullback square 
\[
\xymatrix{
\spic(\Perf_{R,C_p})\ar[r] \ar[d]      & \spic((\Perf_{R})^{BC_p})\ar[d] \\ 
\pi_0\spic(\Perf_{R,C_p})\ar[r] & \pi_0\spic((\Perf_{R})^{BC_p}). 
}
\]
Since, for $\cC = \Perf_R$, both of the paths in the triangle from \Cref{equiv_power_trivial_strict} land in $\spic(\Perf_{R,C_p})$, we obtain a commutative triangle 
\[
\xymatrix{
\spic(\Perf_R) \ar^{p}[d] \ar^{\spic(\Theta^p_\cC)}[rr] & & \spic(\Perf_{R,C_p})     \\
\spic(\Perf_R) \ar_{\pi^*}[rru] & &  
}.
\]
Composing the two paths in the latter triangle with the map $\Perf_{R,C_p} \to \Perf_{R^{tC_p}}$ induced from the \emph{symmetric monoidal} functor $(-)^{tC_p}\colon \Perf_{R,C_p} \to \Perf_{R^{tC_p}}$ we obtain the result. 
\end{proof}

\subsection{Perfect Commutative Ring Spectra}
\label{subsec:pefrect E_infty rings}
We now exploit the commutative triangle in \Cref{fundamental_identity} to study the strict Picard spectrum of the following type of rings. 
\begin{defn}
A $p$-complete commutative ring spectrum $R$ is called \tdef{perfect} if the Tate-valued Frobenius map $\fr^p\colon R\to R^{tC_p}$ is an isomorphism. 
\end{defn}

\begin{rem}
This is a generalization of the notion of \emph{p-perfect} commutative ring spectra from \cite{yuan2022integral}. We are unaware of examples of perfect commutative ring spectra that are not $p$-perfect. 
\end{rem}

Our determination of the strict Picard spectrum of perfect commutative ring spectra essentially originates from the following structural result:
\begin{prop} \label{p_invertible_strict_picard}
Let $R$ be a perfect $p$-complete commutative ring spectrum. Then, the multiplication by $p$ map is an isomorphism on $\spic(R)$: 
\[
p\colon \spic(R)\iso \spic(R). 
\]
\end{prop}

\begin{proof}
The commutative triangle in \Cref{fundamental_identity} implies that $p$ is invertible from the right, with inverse $\can^p \circ (\fr^p)^{-1}$. Since $p$ is central in the endomorphisms of $\spic(R)$, we deduce that it is invertible.  
\end{proof}

This allows us to replace $\ZZ$ with $\ZZ[1/p]$ in the definition of $\spic(R)$. 

\begin{cor} \label{spic_maps_Z_1_p}
Let $R$ be a perfect $p$-complete commutative ring spectrum. Then 
\[
\spic(R)\simeq \hom_{\Sp^\cn}(\ZZ[1/p],\pic(R)). 
\]
\end{cor}

\begin{proof}
Since $p$ is invertible on $\spic(R)$, we obtain:
\begin{align*}
 \spic(R) &\simeq \invlim(\dots\oto{p}\spic(R) \oto{p} \spic(R)) \\
 &\simeq \hom_{\Sp^\cn}(\colim(\ZZ \oto{p}\ZZ \oto{p}\ZZ \oto{p}\dots),\pic(R)) \simeq \hom_{\Sp^\cn}(\ZZ[1/p],\pic(R)). 
\end{align*}
\end{proof}
We turn to compute the strict Picard spectrum of perfect commutative ring spectra.

\begin{thm} 
\label{p_complete_perfect_easy_pic}
Let $R$ be a perfect $p$-complete commutative ring spectrum. Then
$\spic(R)$ is a $1$-truncated connective $\ZZ$-module with the following homotopy groups:
\[
\pi_0\spic(R)\simeq \Hom_\Ab(\ZZ[1/p],\Pic^0(R)) \oplus \Ext^1_{\Ab}(\ZZ[1/p],(\pi_0R)/p^\times)
\]
and
\[
\pi_1\spic(R)\simeq \Hom_\Ab(\ZZ[1/p],(\pi_0R)/p^\times)
\] 
\end{thm}
Note that the homotopy groups determine $\spic(R)$ since it is a $\ZZ$-module spectrum. 
 \begin{proof}
First, we have a fiber sequence in $\Sp$ of the form 
\[
\tau_{\ge 2}\pic(R) \to \pic(R)\to \tau_{\le 1}\pic(R).
\]
The homotopy groups of $\tau_{\ge 2}\pic(R)$ agree with that of $\Sigma\tau_{\ge 1}R$. Since the property of being $p$-complete depends only on the homotopy groups (\cite[Proposition 2.5]{bousfield1979localization}) and since $R$ is $p$-complete, so is $\tau_{\ge 2}\pic(R)$. Since the mapping spectrum from $\ZZ[1/p]$ to a $p$-complete spectrum is $0$, we deduce from \Cref{spic_maps_Z_1_p} that 
\[
\spic(R)\simeq \hom_{\Sp^\cn}(\ZZ[1/p],\pic(R)) \simeq \hom_{\Sp^\cn}(\ZZ[1/p],\tau_{\le 1}\pic(R)) \simeq \hom_{\Mod_\ZZ^\cn}(\ZZ[1/p],\zlin{\tau_{\le 1}\pic(R)}).  
\]

The spectrum $\tau_{\le 1}\pic(R)$ is $1$-truncated, and multiplication by $\eta\in \pi_1\Sph$ on $\pi_0 \pic(R) = \Pic(R)$ is given by $\dim\colon \Pic(R)\to R^\times$  (\Cref{dim_eta}). Combining this with the formula in \Cref{strict_elements_1_truncated},  we deduce that 
\[
\zlin{\tau_{\le 1}\pic(R)} \simeq \Pic^0(R) \oplus \Sigma (\pi_0R)^\times. 
\]
Since $\hom_{\Mod_\ZZ^\cn}(\ZZ[1/p],-)$ computes the higher Ext groups in $\Ab$, we are reduced to show that 
\[
\hom_{\Mod_\ZZ^\cn}(\ZZ[1/p],\Sigma\pi_0R^\times) \simeq \hom_{\Mod_\ZZ^\cn}(\ZZ[1/p],\Sigma\pi_0R/p^\times). 
\]

Since $R$ is $p$-complete, the ring $A:=\pi_0R$ is derived $p$-complete (\cite[Proposition 2.5]{bousfield1979localization}) and hence the reduction mod $p$ map induces a surjection $A^\times \to (A/p)^\times$. Hence, it would suffice to show that the kernel of this map is also derived $p$-complete. 
We have an exact sequence 
\[
(1+p^2A)^\times \to (1+pA)^\times \to (1+pA)^\times/(1+p^2A)^\times,
\]
and $(1+pA)^\times/(1+p^2A)^\times$ is $p$-torsion and hence derived $p$-complete, so $(1+pA)^\times$ is derived $p$-complete if and only if $(1+p^2A)^
\times$ is. 

By Hensel's lemma, the $p$-adic logarithm $\log_p\colon (1+p^2A)^\times \to A$ is injective and its image contains $p^2A$. It follows that the kernel and cokernel of this map are both $p^2$-torsion and hence derived $p$-complete. This implies that $(1+p^2A)^\times$ is derived $p$-complete if and only if $A$ is.
\end{proof}

\subsubsection{Spherical Witt Vectors}
The spherical Witt vectors construction provides a rich supply of perfect commutative ring spectra. 
We shall use \Cref{p_complete_perfect_easy_pic} to compute the strict Picard spectrum of these commutative ring spectra. 
First, we recall the definition. 
\begin{defn} (cf. \cite[Example 5.2.7]{lurie2018elliptic})
For a perfect (ordinary) ring $\kappa$ of characteristic $p$, the commutative ring spectrum $\Sph\WW(\kappa)$ of \tdef{spherical Witt vectors} over $\kappa$ is the connective, $p$-complete  commutative ring spectrum characterized by the following universal property: For every connective $p$-complete   commutative ring spectrum $R$, 
\[
\Map_{\calg(\Sp)}(\Sph\WW(\kappa),R)\simeq \Map_{\calg^\heartsuit_{\FF_p}}(\kappa,(\pi_0R)/p). 
\]
\end{defn}

The spherical Witt vectors have additional properties that allow us to determine their strict Picard spectra.  
\begin{prop}(cf. {{\cite[Example 6.14]{yuan2022integral}}}, {{\cite[Proposition 2.7]{mao2020perfectoid}}}) \label{properties_spherical_witt}
Let $\kappa$ be a perfect ring of characteristic $p$. Then: \begin{enumerate}
\item $\pi_0\Sph\WW(\kappa)\simeq \WW(\kappa)$, the ring of Witt vectors of $\kappa$. 
\item $\FF_p\otimes \Sph\WW(\kappa)\simeq \kappa$. 
\item $\Sph\WW(\kappa)$ is perfect. 
\end{enumerate}
\end{prop}

We begin by computing the (non-strict) Picard group of $\Sph\WW(\kappa)$. 

\begin{prop} 
\label{spherical_Witt_Pic}
 Let $\kappa$ be a perfect ring of characteristic $p$. The quotient map 
 \[
\Sph\WW(\kappa)\to \kappa
 \]
induces an isomorphism 
\[ 
\Pic(\Sph\WW(\kappa)) \iso \Pic(\kappa).
\]
\end{prop}
\begin{proof}
We start by showing that the map $\Pic(\Sph\WW(\kappa)) \to \Pic(\kappa)$
 is surjective, namely, that every $\mathcal{L}\in \Pic(\kappa)$ is isomorphic to $\tilde{\mathcal{L}}\otimes_{\Sph\WW(\kappa)} \kappa$ for some $\tilde{\mathcal{L}}\in \Pic(\Sph\WW(\kappa))$. 
By \Cref{Pic_ordinary_ring}, we may write $\mathcal{L} = \prod_{i=1}^k \Sigma^{n_i} \mathcal{L}_i$ where $\kappa = \prod_i \kappa_i$ and $\mathcal{L}_i\in \Cl(\kappa_i)$. 
Accordingly, $\Sph\WW(\kappa) \simeq \prod_i \Sph\WW(\kappa_i)$. We may therefore replace $\kappa$ by $\kappa_i$ and assume that $\mathcal{L} = \Sigma^n\mathcal{L}_0$ for $\mathcal{L}_0\in \Cl(\kappa)$. Taking the $-n$-th suspension, we may further reduce to the case where $\mathcal{L}\in \Cl(\kappa)$. 

In this case, $\mathcal{L}$ is projective and compact, so it is a retract of a finitely generated free $\kappa$-module $\mathcal{F}$. Thus, there is an idempotent endomorphism $\varepsilon\colon \mathcal{F}\to \mathcal{F}$ such that $\mathcal{L}\simeq \mathcal{F}[\varepsilon^{-1}]$. 
Since $\mathcal{F}$ is a free module, it lifts to a free $\Sph\WW(\kappa)$-module $\tilde{\mathcal{F}}$. Applying Hensel's Lemma, we can lift the endomorphism $\varepsilon$ to an idempotent endomorphism $\tilde{\varepsilon}$ of the $\Sph\WW(\kappa)$-module $\tilde{\mathcal{F}}$. The object $\tilde{\mathcal{L}}:=\tilde{\mathcal{F}}[\tilde{\varepsilon}^{-1}]$ is a dualizable $\Sph\WW(\kappa)$-module which maps under the functor $(-)\otimes_{\Sph\WW(\kappa)}\kappa$ to the invertible $\kappa$-module $\mathcal{L}$. Since $(-)\otimes_{\Sph\WW(\kappa)}\kappa$ is symmetric monoidal and conservative on dualizable $\Sph\WW(\kappa)$-modules, we deduce that $\tilde{\mathcal{L}}$ is invertible and maps to $\mathcal{L}$ under this functor. 

It remains to show that the map $\Pic(\Sph\WW(\kappa)) \to \Pic(\kappa)$ is injective. Let $\mathcal{X}\in \Pic(\Sph\WW(\kappa))$ for which $\mathcal{X}\otimes_{\Sph\WW(\kappa)} \kappa\simeq \kappa$. We wish to show that $\mathcal{X}\simeq \Sph\WW(\kappa)$. By \Cref{properties_spherical_witt}(2), $\kappa\simeq \Sph\WW(\kappa)\otimes \FF_p$. Hence, we may identify the functor $(-)\otimes_{\Sph\WW(\kappa)} \kappa$ with the functor \[(-)\otimes \FF_p\colon \Mod_{\Sph\WW(\kappa)} \to \Mod_{\kappa}
.
\] 
Now, $\mathcal{X}$ is a $p$-complete, bounded below spectrum, and $\FF_p\otimes \mathcal{X} \simeq \kappa$ is connective. It follows from the Hurewizc and the Universal Coefficient theorems that $\mathcal{X}$ is connective and 
\[\kappa \simeq \pi_0(\mathcal{X}\otimes \FF_p)\simeq \pi_0(\mathcal{X})\otimes_{\ZZ} \FF_p.
\]
In particular, the image of $1$ under this isomorphism gives a class $\alpha \in \pi_0(\mathcal{X})\otimes_{\ZZ} \FF_p$, which we can lift to a class $\tilde{\alpha}\in \pi_0\mathcal{X}$. We may view $\tilde{\alpha}$ as a map $\Sph\WW(\kappa) \to \mathcal{X}$ of $\Sph\WW(\kappa)$-modules, which, by construction, induces an isomorphism on $\FF_p$-homology. We deduce from the Hurewizc Theorem that $\mathcal{X}\simeq \Sph\WW(\kappa)$.
\end{proof}

We are ready to compute the strict Picard spectrum of $\Sph\WW(\kappa)$.
\begin{thm} 
\label{strict_picard_spherical_witt}
Let $\kappa$ be a perfect ring of characteristic $p$. Then 
\[
\spic(\Sph\WW(\kappa))\simeq  \Cl(\kappa) \oplus \Sigma \kappa^\times \qin \Mod_\ZZ^\cn. 
\]
Namely, 
\[
\pi_i\spic(\kappa) \simeq \begin{cases}
\Cl(\kappa) & i=0 \\ 
\kappa^\times & i=1 \\ 
0 & i>1. 
\end{cases}
\]
Moreover, the map $\spic(\Sph\WW(\kappa)) \to \pic(\Sph\WW(\kappa))$ induces the inclusion $\Cl(\kappa)\into \Pic(\kappa)$ on $\pi_0$, and the multiplicative lift $[-]\colon \kappa^\times \to \WW(\kappa)^\times$ on $\pi_1$. 
\end{thm}

\begin{proof}
By \Cref{properties_spherical_witt}(3), $\Sph\WW(\kappa)$ is perfect, and hence by \Cref{p_complete_perfect_easy_pic},
\[
\spic(\Sph\WW(\kappa)) \simeq \Hom_{\Ab}(\ZZ[1/p],\pic^0(\Sph\WW(\kappa))) \oplus \Ext^1_\Ab(\ZZ[1/p],\pi_0\Sph\WW(\kappa)/p^\times) \oplus \Sigma\Hom_\Ab(\ZZ[1/p],\pi_0\Sph\WW(\kappa)/p^\times). 
\]
Thus, for the computation of $\spic(\Sph\WW(\kappa))$, it will suffice to prove the following three formulas: 
\begin{enumerate}
    \item$\Hom_\Ab(\ZZ[1/p],\pi_0\Sph\WW(\kappa)/p^\times) \simeq \kappa^\times$.
    \item $\Ext^1_\Ab(\ZZ[1/p],\pi_0\Sph\WW(\kappa)/p^\times) =0$ 
    \item $\Hom_{\Ab}(\ZZ[1/p],\Pic^0(\Sph\WW(\kappa)))\simeq \Cl(\kappa)$ 
\end{enumerate}

Now, by \Cref{properties_spherical_witt}(1),  $\pi_0\Sph\WW(\kappa)/p^\times\simeq \kappa^\times$. Moreover, since $\kappa$ is perfect, $p$ is invertible on $\kappa^\times$. This implies $(1)$ and $(2)$. 

For $(3)$, combining \Cref{spherical_Witt_Pic} and \Cref{Pic_ordinary_ring}, we get: 
\[
\Pic(\Sph\WW(\kappa)) \simeq \Pic(\kappa) \simeq \Cl(\kappa) \oplus C^0(\kappa;\ZZ).  
\]
Now, $C^0(\kappa;\ZZ)$ contains no $p$-divisible elements, and since $\kappa$ is perfect, $p$ is invertible on $\Cl(\kappa)$. We deduce that 
\[
\Hom_{\Ab}(\ZZ[1/p],\pic(\Sph\WW(\kappa))) \simeq \Cl(\kappa).
\]
Thus, to prove $(3)$, it remains to verify that all the elements in $\Cl(\kappa)$ correspond to elements of dimension $1$ in $\Pic(\Sph\WW(\kappa))$. 

Let $\mathcal{L}\in \Pic(\Sph\WW(\kappa))$ which maps to $\Cl(\kappa)$ under the quotient map $\pic(\Sph\WW(\kappa)) \to \pic(\kappa)$. Since, by Hensel's lemma, $\WW(\kappa)$ and $\kappa$ have the same set of idempotents, we have $C^0(\WW(\kappa);\ZZ)\simeq C^0(\kappa;\ZZ)$. 
Since 
$
\mathcal{L}\otimes_{\Sph\WW(\kappa)}\kappa \in \Cl(\kappa),
$
its $C^0(\kappa;\ZZ)$-component vanishes, and hence also
\[
\mathcal{L}\otimes_{\Sph\WW(\kappa)}\WW(\kappa) \in \Cl(\WW(\kappa)).
\] 
By the formula for the dimension in the Picard group of ordinary rings (\Cref{Pic_ordinary_ring}), we deduce that 
\[
\dim(\mathcal{L}\otimes_{\Sph\WW(\kappa)}\WW(\kappa)) = 1 \qin \WW(\kappa)^\times = \pi_0\Sph\WW(\kappa)^\times.
\]
The naturality of the symmetric monoidal dimension implies now that $\dim(\mathcal{L})=1$.

It remains to identify the map 
\[
\pi_*\spic(\Sph\WW(\kappa)) \to \pi_*\pic(\Sph\WW(\kappa)). 
\]
The identification of the map on $\pi_0$ follows from the observation that the map $\pi_0\spic(\Sph\WW(\kappa)) \to \Pic(\Sph\WW(\kappa))$ fits into a commutative diagram 
\[
\xymatrix{
\pi_0\spic(\Sph\WW(\kappa)) \ar^-\sim[r]\ar[rd] & \Hom_\Ab(\ZZ[1/p],\Pic(\Sph\WW(\kappa))) \ar^\sim[r]\ar^{\ev_1}[d] &   \Hom_\Ab(\ZZ[1/p],\Pic(\kappa)) \ar^\wr_{\ev_1}[d] \ar^\sim[r] & \Cl(\kappa) \ar[d]\\ 
&\Pic(\Sph\WW(\kappa))\ar^\sim[r] & \Pic(\kappa) \ar^\sim[r] & \Cl(\kappa) \oplus C^0(\kappa;\ZZ),
}
\]
so that it identifies with the right hand vertical inclusion $\Cl(\kappa) \into \Cl(\kappa) \oplus C^0(\kappa;\ZZ)$. 

Similarly, the identification of the induced map on $\pi_1$ follows from the commutative diagram 
\[
\xymatrix{
\pi_1\spic(\Sph\WW(\kappa))\ar[d] \ar^-\sim[r] & \Hom_\Ab(\ZZ[1/p],\WW(\kappa)^\times)\ar^{\ev_1}[d] \ar^{\mod p}_{\sim}[r] & \Hom_{\Ab}(\ZZ[1/p],\kappa^\times)\ar^-{\ev_1}_-\sim[r] & \kappa^\times  \ar^{[-]}[d] \\ 
\pi_1\pic(\Sph\WW(\kappa))\ar^\sim[r] & \WW(\kappa)^\times\ar@{=}[rr] & & \WW(\kappa)^\times  
}
\]
in which the upper horizontal composition is our identification $\pi_1\spic(\Sph\WW(\kappa)) \simeq \kappa^\times$. 
\end{proof}
We obtain a simple formula for the strict units of spherical Witt vectors. 
\begin{cor}
Let $\kappa$ be a perfect ring of characteristic $p$. Then,
\[
\GG_m(\Sph\WW(\kappa)) \simeq \GG_m(\kappa) = \kappa^\times.
\]
\end{cor}

\begin{proof}
We have $\Omega \spic(\Sph\WW(\kappa))\simeq  \GG_m(\Sph\WW(\kappa))$. Hence,
the result follows from \Cref{strict_picard_spherical_witt} by applying the functor $\Omega$.
\end{proof}
We determine the strict Picard spectrum of the $p$-complete sphere as a special case.
\begin{cor} 
\label{pic_p_complete_sphere}
Let $\Sph_p$ be the $p$-complete sphere. Then
\[
\spic(\Sph_p)\simeq \Sigma \FF_p^\times.
\]
\end{cor}

\begin{proof}
Since $\Sph\WW(\FF_p)\simeq \Sph_p$, and since $\Cl(\FF_p) = 0$,
this follows from \Cref{strict_picard_spherical_witt}. 
\end{proof}

\subsection{Strictly Invertible Spectra}
\label{subsec:trictly Invertible Spectra}
As a final application, we compute the strict Picard spectrum of the sphere spectrum. 
As a result, we prove the vanishing of the strict units spectrum $\GG_m(\Sph)$.

\begin{lem} \label{connected_cover_iso_G_m}
The connected cover maps $\Sigma \Sph^\times \to \pic(\Sph)$ and $\Sigma \Sph_p^\times \to \pic(\Sph_p)$ induce isomorphisms
\[
\spic(\Sph) \simeq \zlin{(\Sigma \Sph^\times)} 
\]
and, for every prime $p$,
\[
\spic(\Sph_p) \simeq \zlin{(\Sigma \Sph_p^\times)}.
\]
\end{lem}

\begin{proof}
Since $\Pic(\Sph_p)\simeq \ZZ$, the fiber sequence 
\[
\xymatrix{
\Sigma \Sph_p^\times\ar[r] & \ar[r] \pic(\Sph_p) & \ZZ
}
\]
induces a fiber sequence 
\[
\xymatrix{
\zlin{(\Sigma \Sph_p^\times)} \ar[r] & \spic(\Sph_p) \ar[r] & \hom_{\Sp^\cn}(\ZZ,\ZZ).
}
\]
Note that $\hom_{\Sp^\cn}(\ZZ,\ZZ)\simeq \ZZ$. 
Now, by \Cref{pic_p_complete_sphere}, we have $\spic(\Sph_p) \simeq \Sigma \FF_p^\times$. Since the only map $\Sigma \FF_p^\times \to \ZZ$ is the zero map, the right-hand map in the above fiber sequence has to be the zero map. Since $\ZZ$ is discrete it follows that the left hand map, $\hom_{\Sp^\cn}(\ZZ,\Sigma \Sph_p^\times) \to \spic(\Sph_p)$, is an isomorphism of connective spectra.    

The case of $\Sph$ follows similarly, once we show that the map $\spic(\Sph) \to \hom_{\Sp^\cn}(\ZZ,\Pic(\Sph)) \simeq \ZZ$ is a zero map. But this follows immediately from the commutativity of the diagram 
\[
\xymatrix{
\spic(\Sph)\ar[d]\ar[r] & \spic(\Sph_p) \ar[d] \\
\Pic(\Sph) \ar^\sim[r] \ar^\wr[d]        & \Pic(\Sph_p)\ar^\wr[d]\\ 
\ZZ \ar^\sim[r]                           & \ZZ
}
\]
\end{proof}

We are ready to compute the strict Picard spectrum of $\Sph$. 
\begin{thm}
\label{strict_picard_sphere}
Let $\widehat{\ZZ} = \prod_{p}\ZZ_p$ be the profinite completion of the integers. 
Then,
\[
\spic(\Sph) \simeq \widehat{\ZZ}. 
\]
In other words, 
\[
\pi_i \spic(\Sph) \simeq \begin{cases} \widehat{\ZZ} & i =0 \\ 0 & i>0. \end{cases}
\]
\end{thm}

\begin{proof}
Let $\mathbb{A} := \prod_p \ZZ_p \otimes \QQ$ be the ring of finite ad\'{e}les. 
Consider the arithmetic fracture square (see \cite[\S 8]{bousfield1972homotopy})   
\[
\xymatrix{
\Sph\ar[r]\ar[d] & \QQ \ar[d]\\
\prod_{p}\Sph_p \ar[r] & \mathbb{A} 
}
\]
which is a pullback square in $\calg(\Sp)$. 
Applying the limit preserving functor $(-)^\times\colon \calg(\Sp^\cn) \to \Sp^\cn$, we obtain a pullback square
\[
\xymatrix{
\Sph^\times\ar[r]\ar[d] & \QQ^\times \ar[d]\\
\prod_{p}\Sph_p^\times \ar[r] & \mathbb{A}^\times. 
}
\]
We claim that this square remains a pullback after applying the functor $\Sigma$. Indeed, this follows from the fact that the map 
$\prod_p\ZZ_p^\times \times \QQ^\times \to \mathbb{A}^\times$, at the edge of the corresponding long exact sequence of homotopy groups, is surjective\footnote{Note that this surjectivity is equivalent to the classical fact that $\Cl(\ZZ) = 0$}. 
Thus, we got that the square
\[
\xymatrix{
\Sigma \Sph^\times\ar[r]\ar[d] & \Sigma \QQ^\times \ar[d]\\
 \prod_{p}\Sigma\Sph_p^\times \ar[r] & \Sigma \mathbb{A}^\times. 
}
\]
is again a pullback square in $\Sp^\cn$. Applying the functor  $\zlin{(-)}:=\hom_{\Sp^\cn}(\ZZ,-)$, we obtain a pullback square 
\begin{equation} 
\label{arithmatic_fracture_spic_sphere}
\xymatrix{
\zlin{(\Sigma \Sph^\times)}\ar[r]\ar[d] &  \zlin{(\QQ^\times)} \ar[d]\\
 \prod_{p}\zlin{(\Sigma\Sph_p^\times)} \ar[r] &  \zlin{(\Sigma\mathbb{A}^\times)} 
}
\end{equation}
in $\Mod_\ZZ^\cn$.
Combining \Cref{pic_p_complete_sphere} and \Cref{connected_cover_iso_G_m} we deduce that 
\[
\zlin{(\Sigma \Sph_p^\times)} \simeq \spic(\Sph_p) \simeq \Sigma \FF_p^\times
\] and that
\[
\zlin{(\Sigma \Sph^\times)} \simeq \spic(\Sph).
\] 
Also, using \Cref{strict_elements_1_truncated}, the two other vertices are:
\[
\zlin{(\Sigma\QQ^\times)} \simeq \Sigma \QQ^\times
\] 
and 
\[
\zlin{(\Sigma\mathbb{A}^\times)} \simeq \Sigma \mathbb{A}^\times.
\]
Substituting these computations back into the square (\ref{arithmatic_fracture_spic_sphere}), we obtain the pullback square 
\[
\xymatrix{
\spic(\Sph)\ar[r]\ar[d] & \Sigma \QQ^\times \ar[d]\\
 \prod_{p} \FF_p^\times \ar[r] & \Sigma\mathbb{A}^\times. 
}
\]
From the associated long exact sequence of homotopy groups, we get the exact sequence  
\[
0 \to \pi_1\spic(\Sph) \to \QQ^\times \oplus \prod_p \FF_p^\times \oto{\psi} \mathbb{A}^\times \to \pi_0\spic(\Sph) \to 0,
\]
Here, the map $\psi$ can be computed as follows:
Let 
\[
[-]_p \colon \FF_p^\times \into \ZZ_p^\times 
\] 
denote the multiplicative lift. Then 
\[
\psi(q,a_2,a_3,a_5,\dots)=
(q[a_2]_2^{-1},q[a_3]_3^{-1},q[a_5]_5^{-1}\dots).
\]
Indeed, this follows directly from the fact that $\QQ^\times$ is embedded diagonally in $\mathbb{A}^\times \subseteq \prod_p \QQ_p^\times$, and from the identification of $\pi_1\spic(\Sph_p) \to \pi_1\pic(\Sph_p)\simeq \ZZ_p^\times$ with the multiplicative lift map (see the end of \Cref{strict_picard_spherical_witt}).  
Thus, it remains to show that:
\begin{itemize}
    \item[(a)] $\mathrm{Ker}(\psi)= 0$.
    \item[(b)] $\mathbb{A}^\times/\im(\psi)\simeq \widehat{\ZZ}$.
\end{itemize}

For $(a)$, if $\psi(q,a_2,a_3,a_5,\dots) = 1$ then, in particular, $q$ is a $p$-adic integer for every prime $p$, and hence $q=\pm 1$. If $q=1$ then $[a_p]_p = 1$ for every prime $p$ and hence $a_p = 1$ as well. 
If, on the other hand, $q=-1$ then we must have $[a_2]_2 = -1$ which is impossible since $[-]_2$ is a trivial homomorphism. 

For $(b)$, note that 
\[
\mathbb{A}^\times / \im(\psi) \simeq (\prod_p \ZZ_p^\times)/(\{\pm 1\} \cdot \prod_p [\FF_p^\times]_p).
\]
Set
\[
M := (1 + 4\ZZ_2)^\times \times \prod_{p \text{ odd }} (1 + p \ZZ_p)^\times 
\]
and 
\[
N := \{\pm 1\} \times \prod_{p \text{ odd }} [\FF_p^\times]_p. 
\]
Then, $M$ is a fundamental domain for the translation $N$-action on $\prod_p \ZZ_p^\times$, namely,  $M\cap N = 0$ and $M\cdot N = \prod_p \ZZ_p^\times$. Consequently, 
\[
\mathbb{A}^\times / \im(\varphi) \simeq (\prod_p \ZZ_p^\times) / N \simeq M.  
\]
Finally, the $p$-adic logarithms provide isomorphisms $
(1 + 4 \ZZ_2)^\times \simeq \ZZ_2
$
and 
$
(1 + p \ZZ_p)^\times \simeq \ZZ_p
$
for every odd prime $p$,
so that 
\[
M \simeq \prod_p \ZZ_p \simeq \widehat{\ZZ}
\]
and the result follows.
\end{proof}

\begin{cor}
The spectrum of strict units of the sphere spectrum is trivial:
\[
\GG_m(\Sph)\simeq 0. 
\]
\end{cor}

\bibliographystyle{alpha}
\phantomsection\addcontentsline{toc}{section}{\refname}
\bibliography{strict}

\end{document}